\documentclass[12pt]{amsart}

\usepackage[top=16mm, bottom=16mm, left=44.45mm, right=44.45mm]{geometry}

\usepackage[pagewise]{lineno}\nolinenumbers

\usepackage{amsmath, amsfonts, amssymb, amsthm}
\usepackage{amsrefs}
\usepackage{mathrsfs}
\usepackage{enumitem} 
\usepackage{hyperref}
\usepackage{color}
\usepackage{dsfont}
\usepackage{csquotes}
\usepackage{epigraph}
\usepackage{comment}
\usepackage{mlmodern}
\usepackage[T1]{fontenc}

\usepackage{color}
\definecolor{darkblue}{rgb}{0.0,0.0,0.3}
\hypersetup{
    colorlinks=false,
    linkcolor=blue,
    urlcolor=darkblue,
    }

\newtheorem{theorem}{Theorem}[section]

\newtheorem{lemma}[theorem]{Lemma}

\theoremstyle{definition}

\newtheorem{definition}[theorem]{Definition}

\newtheorem{notation}[theorem]{Notation}

\theoremstyle{remark}
\newtheorem{remark}[theorem]{Remark}

\numberwithin{equation}{section}

\newcommand{\bR}{\mathbb{R}}

\newcommand{\Ampere}{Amp\`{e}re}

\newcommand{\Garding}{G\r{a}rding}

\DeclareMathOperator{\diam}{diam}

\DeclareMathOperator{\const}{const}

\title[Rigidity of entire solutions]{A Liouville type rigidity result for entire solutions of the $k$-Hessian equation}
\author{Bin Wang}
\address[]{Institute for Theoretical Sciences, Westlake University, Hangzhou, China} 
\thanks{This work was supported by the China Postdoctoral Science Foundation under Grant Number 2026M793406}
\email{wangbin@westlake.edu.cn}

\begin{document}
\begin{abstract}
In this note, we prove a Liouville type rigidity result for smooth entire solutions of the $k$-Hessian equation. The proof makes use of a concavity inequality for the $k$-Hessian operator.

\end{abstract}
\maketitle
\setcounter{tocdepth}{1} 


\section{Introduction}
In this note, we study the rigidity of smooth entire solutions to the following fully nonlinear, non-uniformly elliptic equation,
\begin{equation}
F(D^2u):=\sigma_k(\lambda(D^2u))=1 \quad \text{in $\bR^n$}, \label{the equation}
\end{equation} where $\lambda(D^2u)=(\lambda_1,\ldots,\lambda_n)$ are the eigenvalues of the Hessian $D^2u$ and 
\[\sigma_k(\lambda):=\sum_{1 \leq i_1<\cdots<i_k \leq n} \lambda_{i_1}\cdots \lambda_{i_k}\] is the $k$-th elementary symmetric polynomial. Following Wang \cite{Wang}, we call it \textit{the $k$-Hessian equation}. 

The $k$-Hessian equation has been extensively studied since the seminal work \cite{CNS} of Caffarelli, Nirenberg, and Spruck. The $\sigma_k$ operator has a rich structure and interpolates between
\begin{gather*}
\text{the Laplacian }\sigma_1\bigl(\lambda(D^2u)\bigr)
  = \sum_{i=1}^{n}\lambda_i
  = \Delta u, \\[4pt]
\text{and the Monge--Ampère operator }\sigma_n\bigl(\lambda(D^2u)\bigr)
  = \prod_{i=1}^{n}\lambda_i
  = \det(D^2u).
\end{gather*} The $\sigma_2=1$ equation and the quotient equation $\sigma_3/\sigma_1=1$ are also equivalent to the special Lagrangian equation
\[\sum_{i=1}^{n} \arctan \lambda_i(D^2u)=\Theta\] in low dimensions for some particular values of the ``phase'' $\Theta$. Moreover, there are several geometric problems that can be reduced to solving some $\sigma_k(\lambda(A[u]))$ equations with $\lambda(A[u])$ being eigenvalues of more complicated matrices $A[u]$.

The $k$-Hessian operator is elliptic at $k$-admissible functions, where the ellipticity is interpreted in the sense that 
\[\left\{\frac{\partial F}{\partial u_{ij}}\right\} \quad \text{is positive definite}.\]

\begin{definition}
Let $\Omega$ be an open connected subset of $\bR^n$. We say a function $u \in C^2(\Omega) \cap C(\overline{\Omega})$ is $k$-admissible in $\Omega$ if 
\[\lambda(D^2u) \in \Gamma_k:=\{\lambda \in \bR^n: \sigma_j(\lambda)>0 \quad \forall\ 1 \leq j \leq k\}\]  for all $x \in \Omega$.
\end{definition}

It is known \cite{Warren} that, when $2k \leq n+1$, not every smooth entire $k$-admissible solution of the $k$-Hessian equation \eqref{the equation} is a polynomial. In other words, the Bernstein type rigidity does not hold in general. In this note, we prove the following Liouville type rigidity result for the $k$-Hessian equation, when $2k \geq n+1$.

\begin{theorem} \label{the theorem}
Let $n \geq 3$ and $2 \leq k \leq n-1$ be such that $2k>n$. Suppose $u \in C^{\infty}(\bR^n)$ is a smooth function such that
\begin{equation}
\sigma_{k}\bigl(\lambda(D^2u)\bigr)=1 \quad \text{and} \quad \lambda(D^2u) \in \Gamma_k \quad \text{in $\bR^n$}.\label{basic assumption}
\end{equation} If there exist some constants $b,B>0$ such that
\begin{equation}
u(x) \geq b|x|^2-B \quad \text{for all $x \in \bR^n$},\label{quadratic growth}
\end{equation} then $u$ must be a quadratic polynomial. 
\end{theorem}

\begin{remark}
Note that our theorem holds for $2k \geq n+1$.
In particular, reflecting on the dimensional range $2k \leq n+1$ for the existence of non-polynomial solutions, it seems necessary to impose an additional hypothesis like \eqref{quadratic growth} beyond the basic one \eqref{basic assumption}, at least for the borderline case $2k=n+1$.
\end{remark}
For the proof of Theorem \ref{the theorem}, it suffices to establish a Pogorelov type interior estimate,
\begin{equation}
\sup_{\Omega}\ (-u)^{\beta}|D^2u| \leq C, \label{the Pogorelov estimate}
\end{equation} for $k$-admissible solutions $u$ to the following Dirichlet problem:
\[\begin{alignedat}{2}
\sigma_{k}\bigl(\lambda(D^2u)\bigr)&=1 &\quad & \text{in a bounded domain $\Omega \subseteq \bR^n$},\\
u&=0 &\quad & \text{on $\partial \Omega$}.\\
\end{alignedat}\] A crucial thing to note is that the constants $\beta,C>0$ in \eqref{the Pogorelov estimate} cannot depend on $\sup_{\Omega} |Du|$. In other words, we are not allowed to add the usual gradient term
\begin{equation}
\frac{1}{2}|Du|^2 \label{the gradient term}
\end{equation} to the auxiliary function and this is the major difficulty for obtaining the estimate \eqref{the Pogorelov estimate}.

In order to obtain \eqref{the Pogorelov estimate}, we will make use of a concavity inequality for the $\sigma_k$ operator that was conjectured by Ren-Wang \cite{Ren-Wang-2} and has recently been verified by Yan \cite{Yan}. Once we have the required estimate \eqref{the Pogorelov estimate} for some $C>0$ that is independent of $\sup_{\Omega}|Du|$, the proof of Theorem \ref{the theorem} would be rather routine; see \cite{Li-Ren-Wang, Chen-Xiang, Chu-Dinew, Tu, Zhang}.

We now review some related literature. Let us first discuss the ``end-point'' cases $k=1$ and $k=n$ for which Theorem \ref{the theorem} holds without the quadratic growth condition \eqref{quadratic growth}. Indeed, by the classical Liouville theorem for harmonic functions, one can deduce that every smooth entire solution of $\Delta u =1$ that is semi-convex in the sense that $D^2u \geq -KI$ for some constant $K>0$ must be a quadratic polynomial. Similarly, for the Monge-\Ampere\ equation, every smooth entire solution of $\det(D^2u)=1$ with $D^2u>0$ must be a quadratic polynomial. This is a classical result proved by J\"{o}rgens \cite{Jorgens} in dimension $n=2$, by Calabi \cite{Calabi} in dimensions $3 \leq n \leq 5$, and by Pogorelov \cite{Pogorelov} in all dimensions; see also Cheng-Yau \cite{Cheng-Yau} for a geometric proof and Caffarelli-Li \cite{Caffarelli-Li} for an extension to viscosity solutions.

For $k=2$, Chen and Xiang \cite{Chen-Xiang} proved the estimate \eqref{the Pogorelov estimate} in dimension $n=3$, by employing a concavity inequality due to Qiu \cite{Qiu}. In the same paper, by using a concavity inequality due to Guan-Qiu \cite{Guan-Qiu}, they also proved that the estimate \eqref{the Pogorelov estimate} holds for $\sigma_2$ in higher dimensions if $u$ additionally satisfies the condition $\sigma_{3}(\lambda(D^2u)) \geq -A$ for some constant $A>0$. Consequently, Theorem \ref{the theorem} follows in the corresponding cases.

In \cite{Tu}, Tu deduced a concavity inequality for the $\sigma_{n-1}$ operator building on the work of Lu-Tsai \cite{LT-PAMS}, Ren-Wang \cite{Ren-Wang-1}, and Zhang \cite{Zhang}. As a consequence, Tu was able to prove the estimate \eqref{the Pogorelov estimate} and so Theorem \ref{the theorem} follows for $k=n-1$, where $n \geq 3$. In particular, Tu's result recovers Chen-Xiang's result \cite[Theorem 1.4]{Chen-Xiang} for $\sigma_2$ in dimension $n=3$.

For the intermediate cases $3 \leq k \leq n-2$, Bao-Chen-Guan-Ji \cite{BCGJ} proved Theorem \ref{the theorem} for solutions $u$ with $\lambda(D^2u) \in \Gamma_n$, which are solutions with $D^2u>0$. In \cite{Li-Ren-Wang}, Li-Ren-Wang improved the result to solutions $u$ with $\lambda(D^2u) \in \Gamma_{k+1}$. Chu and Dinew \cite{Chu-Dinew} further relaxed the condition to $\lambda(D^2u) \in \Gamma_k$ and $\sigma_{k+1}(\lambda(D^2u)) \geq -A$ for some constant $A>0$. In particular, Chu-Dinew's result generalizes Chen-Xiang's result \cite[Theorem 1.3]{Chen-Xiang} for $\sigma_2$ with $\sigma_3 \geq -A$.
Recently, Zhang \cite{Zhang} proved the result for solutions with $\lambda(D^2u) \in \Gamma_k$ and $D^2u \geq -KI$ for some constant $K>0$.

\begin{remark}
Our Theorem \ref{the theorem} removes Zhang's \cite{Zhang} semi-convexity condition when $2k>n$ and generalizes Tu's result \cite{Tu} from $k=n-1$ to $2k>n$. In particular, our result is new for $k=n-2$, $n \geq 5$.
\end{remark}
\begin{remark}
It might be worth recalling the following inclusion relations,
\begin{equation}
\Gamma_n \subseteq \cdots \subseteq \Gamma_{k+1} \subseteq \Gamma_{k} \subseteq \cdots \subseteq \Gamma_1. \label{inclusion relation}
\end{equation} In particular,
\[\Gamma_n=\{\lambda \in \bR^n: \lambda_i>0 \quad \forall\ 1 \leq i \leq n\}\] and so
\[\{u: \lambda(D^2u) \in \Gamma_n\}=\{u: D^2u>0\}.\]
Moreover, on the set
\[\Gamma_k \cap \{\lambda: \sigma_k(\lambda)=1\},\] the condition $\sigma_{k+1}(\lambda) \geq -A$ for some constant $A>0$ implies that there is some constant $K=K(n,k,A,\sigma_k)>0$ such that $\lambda_i \geq -K$ for all $1 \leq i \leq n$; see \cite[Lemma 2.2]{Zhang}. 
\end{remark}

\begin{remark}
We shall also comment on the proofs for the intermediate cases $3 \leq k \leq n-2$. Bao-Chen-Guan-Ji \cite{BCGJ} made use of the Legendre transform; Li-Ren-Wang \cite{Li-Ren-Wang} devised a novel choice of the auxiliary function which enabled them to get the estimate \eqref{the Pogorelov estimate} for $(k+1)$-admissible solutions; Chu-Dinew \cite{Chu-Dinew} used the condition $\sigma_{k+1} \geq -A$ to derive a Pogorelov type interior gradient estimate for admissible solutions so that they could add a gradient dependence \eqref{the gradient term} to their auxiliary function which would enable them to get the estimate \eqref{the Pogorelov estimate}. Zhang \cite{Zhang} derived a strengthened version of Lu's \cite{Lu-CVPDE} concavity inequality under the condition that $D^2u \geq -KI$; the inequality would then lead to the estimate \eqref{the Pogorelov estimate}. The ideas of Lu's weaker inequality and Li-Ren-Wang's novel auxiliary function both come from the work \cite{Guan-Ren-Wang} of Guan-Ren-Wang.
\end{remark}

For the $\sigma_2$ equation, there is another Liouville type rigidity result due to Shankar-Yuan \cite{SY-Duke}: Every smooth entire solution of $\sigma_2(D^2u)=1$ that is semi-convex must be a quadratic polynomial. Note that this result assumes the semi-convexity condition instead of the quadratic growth condition \eqref{quadratic growth}. In \cite{CY, BCGJ}, it was conjectured that this result should remain valid for the $\sigma_k$ equation as well. However, at the time of writing this manuscript, the conjecture has not been verified.

Finally, we conclude by mentioning rigidity results for some related equations. In a recent preprint \cite{Lu-Sroka}, Lu and Sroka made a keen observation that solutions to the quotient equation $\sigma_2/\sigma_1=1$ can be transformed into solutions of the $\sigma_2$ equation. Therefore, by applying the rigidity result of Shankar-Yuan \cite{SY-Duke}, it follows that every smooth entire solution of $\sigma_2/\sigma_1=1$ that is semi-convex must be a quadratic polynomial. In \cite{Li-Wu}, Li and Wu proved Theorem \ref{the theorem} for smooth entire solutions with $\lambda(D^2u) \in \Gamma_n$ to the quotient $\sigma_k/\sigma_l$ under the same quadratic growth condition \eqref{quadratic growth}, where $l=k-1,k-2$. Their proof is through a combination of the constant rank theorem, interior Hessian estimates, and the Legendre transform.

In \cite{Liu-Ren}, Liu and Ren proved Theorem \ref{the theorem} for $k$-admissible solutions to the sum Hessian equation $\sigma_k+\alpha\sigma_{k-1}=1$. Their result holds for all $k$ and the parameter $\alpha$ being positive played a crucial role in their proof. It parallels Li-Ren-Wang's result \cite{Li-Ren-Wang} for $(k+1)$-admissible solutions of $\sigma_k=1$ because the natural ellipticity cone for the sum Hessian equation is $\Gamma_{k-1}$.
In \cite{Mei-Yan}, by modifying the argument of Chang-Yuan \cite{CY}, Mei and Yan proved that every smooth entire solution of the sum Hessian equation $\sigma_3+\sigma_2=\const$ with $\lambda(D^2u) \in \Gamma_n$ must be a quadratic polynomial, without imposing \eqref{quadratic growth}. 

The rest of this note is organized as follows. In Section \ref{Preliminaries}, we introduce some notations and state some properties of the $\sigma_k$ operator. In Section \ref{the theorem proof}, we derive the required estimate \eqref{the Pogorelov estimate}. The proof of Theorem \ref{the theorem} would then follow as a consequence.

\section{Preliminaries} \label{Preliminaries}

For a symmetric matrix $A=(a_{ij})$ and a function
\[F: \{\text{symmetric matrices}\} \to \bR,\] we define
\begin{equation}
F^{ij}(A)=\frac{\partial F}{\partial a_{ij}}, \quad F^{ij,rs}(A)=\frac{\partial^2 F}{\partial a_{ij} \partial a_{rs}}.\label{derivatives of F}
\end{equation}
When $F$ is of the form
\[F(A)=f(\lambda(A))\] for some symmetric function $f$ of $n$ variables, where $\lambda(A)=(\lambda_1,\ldots,\lambda_n)$ are the eigenvalues of $A$, the function $F$ is as smooth as $f$ and is concave if $f$ is concave. When $A$ is diagonal, we have $F^{ij}=f_i\delta_{ij}$ where
\begin{equation}
f_i=\frac{\partial f}{\partial \lambda_i}.\label{derivatives of f}
\end{equation} Moreover, we have
\[\sum_{i,j} F^{ij}a_{ij}=\sum_{i=1}^{n} f_i(\lambda(A))\lambda_i, \quad \sum_{i,j,k} F^{ij}a_{ik}a_{jk}=\sum_{i=1}^{n} f_i(\lambda(A))\lambda_{i}^2.\]

Recall the definition of the $k$-th elementary symmetric polynomial,
\[\sigma_k(\lambda_1,\ldots,\lambda_n):=\sum_{1 \leq i_1<\cdots<i_k \leq n} \lambda_{i_1}\cdots \lambda_{i_k},\] and we adopt the convention that $\sigma_0:=1$ and $\sigma_k:=0$ for $k>n$. The associated $k$-th \Garding\ cone is an open symmetric convex cone defined by 
\[\Gamma_k:=\{\lambda \in \bR^n: \sigma_j(\lambda)>0 \quad \forall\ 1 \leq j \leq k\}.\]

\begin{notation}
Observe that
\[\frac{\partial}{\partial \lambda_i}\sigma_k(\lambda)=\sigma_{k-1}(\lambda)\bigg|_{\lambda_i=0}=\sigma_{k-1}(\lambda_1,\ldots,\lambda_{i-1},0,\lambda_{i+1},\ldots,\lambda_n).\] Therefore, we may use $\sigma_{k-1}(\lambda|i)$ to denote the first order derivatives of $\sigma_k(\lambda)$. The notation $\sigma_{k-2}(\lambda|ij)$ is defined in the same fashion for the second order derivatives. 
\end{notation}

The following are some commonly used properties of the $\sigma_k$ operator and we may state them without proofs.
\begin{lemma} \label{sigma_k properties 1}
For all $1 \leq k \leq n$ and $\lambda \in \bR^n$, we have
\begin{align*}
&\ \sigma_k(\lambda)=\lambda_i\sigma_{k-1}(\lambda|i)+\sigma_{k}(\lambda|i) \quad \text{for any $1 \leq i \leq n$},\\
&\ \sum_{i=1}^{n}\lambda_i\sigma_{k-1}(\lambda|i)=k\sigma_k(\lambda), \\
&\ \sum_{i=1}^{n} \sigma_{k-1}(\lambda|i)=(n-k+1)\sigma_{k-1}(\lambda).
\end{align*}
\end{lemma}

\begin{lemma} \label{sigma_k properties 2}
Let $\lambda \in \Gamma_k$ be ordered as 
\[\lambda_1 \geq \cdots \geq \lambda_n.\] Then $\lambda_k>0$ and 
\[0<\frac{\partial\sigma_k}{\partial \lambda_1} (\lambda) \leq \cdots \leq \frac{\partial \sigma_k}{\partial \lambda_n}(\lambda).\] If $\lambda_i \leq 0$, then
\[\lambda_i>-\frac{n-k}{k} \lambda_1.\]
\end{lemma}

\begin{lemma} \label{NM}
Let $n \geq k>l \geq 0$ and $n \geq r>s \geq 0$. Suppose $\kappa \in \Gamma_k$. If $k \geq r$ and $l \geq s$, then
\[\left[\frac{\binom{n}{k}^{-1}\sigma_k(\kappa)}{\binom{n}{l}^{-1}\sigma_l(\kappa)}\right]^{\frac{1}{k-l}} \leq \left[\frac{\binom{n}{r}^{-1} \sigma_r(\kappa)}{\binom{n}{s}^{-1}\sigma_s(\kappa)}\right]^{\frac{1}{r-s}}.\]
\end{lemma}

\begin{notation}
The $\sigma_k$ operator can be regarded as a function on the space of $n \times n$ symmetric matrices, by writing $\sigma_k(A)$ to mean $\sigma_k(\lambda(A))$. In particular, $\sigma_k(A)$ is the sum of the $k \times k$ principal minors of $A$.
With this abuse of notation, we may also write $\sigma_{k}^{ij}$ and $\sigma_{k}^{ij,kl}$ to mean the first and the second order derivatives in the sense of \eqref{derivatives of F}. When the matrix is diagonal, the notations $\sigma_{k}^{ii}$ and $\sigma_{k}^{ii,jj}$ can be used to denote the derivatives in the sense of \eqref{derivatives of f} without causing confusion. 
\end{notation}

When $A=(a_{ij})=(\lambda_i\delta_{ij})$ is diagonalized with eigenvalues $\lambda_1 \geq \cdots \geq \lambda_n$, we have
\[
\sigma_k^{pq}(A)
=
\frac{\partial \sigma_k}{\partial \lambda_p}(\lambda)\delta_{pq}
=
\sigma_{k-1}(\lambda\mid p)\delta_{pq},
\]

\begin{equation}
\sigma_k^{pq,rs}(A)
=
\begin{cases}
\displaystyle
\sigma_{k-2}(\lambda\mid pr),
&
p=q,\ r=s,\ p\neq r,
\\[10pt]
\displaystyle
-\sigma_{k-2}(\lambda\mid pq),
&
p=s,\ q=r,\ p\neq q,
\\[10pt]
0,
&
\text{otherwise}.
\end{cases} \label{2nd derivatives for sigma k}
\end{equation} Moreover, by Lemma \ref{sigma_k properties 1} and Lemma \ref{sigma_k properties 2}, one can verify that
\begin{equation}
\sigma_{k}^{pp,qq}=\frac{\sigma_{k}^{pp}-\sigma_{k}^{qq}}{\lambda_q-\lambda_p} \geq 0 \quad \text{if $\lambda_p \neq \lambda_q$}.\label{2nd derivative for sigma k 2}
\end{equation}

\begin{lemma} \label{the concavity inequality}
Let $n \geq 3$ and $2 \leq k \leq n-1$. Assume the vector $\lambda \in \Gamma_k$ is ordered with multiplicity $m \geq 2$ in the sense that
\[\lambda_1 = \cdots = \lambda_m > \lambda_{m+1} \geq \cdots \geq \lambda_n.\] For 
\begin{equation}
0<\gamma<\min \left\{\frac{2k}{n}, 1+ \frac{2k-n}{2k^2+n}\right\},\label{gamma}
\end{equation} there exists some $\eta_{*}=\eta_{*}(n,k,\gamma)>0$ such that if 
\[\lambda_1 \geq \frac{\sigma_{k}^{1/k}(\lambda)}{\eta_{*}}>0,\] then
\begin{gather} \label{concavity}
\begin{split}
& -\sum_{p \neq q} \frac{\sigma_{k}^{pp,qq} \xi_p\xi_q}{\lambda_1} - \gamma \frac{\sigma_{k}^{11} \xi_{1}^2}{\lambda_{1}^2}+2\sum_{i>m} \frac{\sigma_{k}^{ii}\xi_{i}^2}{\lambda_1(\lambda_1-\lambda_i)}\\
\geq &\ - \frac{2}{\lambda_1 \sigma_k} \left(\sum_{i=1}^{n} \sigma_{k}^{ii}\xi_{i}\right)^2
\end{split}
\end{gather} for any $\xi \in \bR^n$ with
\[\xi_2 = \cdots = \xi_m = 0.\] If $m=1$, then the inequality \eqref{concavity} holds for all $\xi \in \bR^n$.
\end{lemma}
\begin{proof}
See \cite[Theorem 1.1 and Corollary 2.3]{Yan}.
\end{proof}
\section{A Pogorelov type interior estimate} \label{the theorem proof}

In this section, we prove the following Pogorelov type interior estimate that is independent of $\sup_{\Omega}|Du|$, from which Theorem \ref{the theorem} follows routinely.

\begin{lemma}
Let $\Omega \subseteq \bR^n$ be a smooth bounded domain. Suppose $u \in C^{\infty}(\Omega) \cap C^2(\overline{\Omega})$ is a $k$-admissible solution of 
\[\begin{alignedat}{2}
\sigma_{k}\bigl(\lambda(D^2u)\bigr)&=1 &\quad & \text{in $\Omega$},\\
u&=0 &\quad & \text{on $\partial \Omega$}.\\
\end{alignedat}\] Then there exist constants $\beta, C>0$ whose values depend only on $n$, $k$, and $\diam(\Omega)$ such that
\[\sup_{\Omega}\ (-u)^\beta |D^2u| \leq C.\]
\end{lemma}
\begin{proof}
By translation if necessary, we may assume $0 \in \Omega$.
Consider the auxiliary function
\[Q=(-u)^{\beta}\lambda_{\max}(D^2u)\exp\left(\frac{1}{2}|x|^2\right),\] where $\lambda_{\max}(D^2u)$ is the largest eigenvalue of $D^2u$ and $\beta>0$ is a constant whose value is to be determined later.

By the inclusion relation \eqref{inclusion relation}, we see that $\Delta u > 0$ and so $u<0$ in $\Omega$. On the other hand, there exist \cite[Proposition 4.1]{Chu-Dinew} constants $c_1,c_2>0$ whose values depend only on $n$, $k$, and $\diam(\Omega)$ such that the function
\[\underline{u}:=\frac{c_1}{2}|x|^2-c_2\] satisfies
\[\begin{alignedat}{2}
\sigma_{k}\bigl(\lambda(D^2\underline{u})\bigr)&=1 &\quad & \text{in $\Omega$},\\
\underline{u}&\leq 0 &\quad & \text{on $\partial \Omega$}.\\
\end{alignedat}\] It follows that
\[\underline{u} \leq u < 0 \quad \text{in $\Omega$}.\] Therefore, we may allow dependence on $\sup_{\Omega} |u|$ for constants in the subsequent computations. 

Suppose $Q$ attains its maximum at some interior $x_0 \in \Omega$. We may choose suitable coordinates around $x_0$ such that the Hessian $D^2u(x_0)=\lambda_i(x_0)\delta_{ij}$ is diagonal and
\[\lambda_{\max}(x_0)=\lambda_1(x_0) \geq \cdots \geq \lambda_n(x_0).\] Let $C>0$ denote a positive constant whose value depends only on $n$, $k$, $\diam(\Omega)$ and $\sup_{\Omega} |u|$. The value of $C$ may change from line to line but will still be denoted by the same symbol; this would not cause confusion because the exact magnitude would be irrelevant in our analysis but the dependence on the known data. In our derivation below, given any such a constant $C>0$, we may always assume $\lambda_1(x_0) \geq C$ is sufficiently large, otherwise the desired bound would follow immediately. Note that we would impose this assumption without explicitly saying so.

In case $\lambda_1$ has multiplicity $m \geq 2$ in the sense that
\[\lambda_1(x_0)=\cdots=\lambda_m(x_0)>\lambda_{m+1}(x_0)\geq \cdots \geq \lambda_n(x_0),\] we consider the function
\[\varphi(x):=Q(x_0)(-u)^{-\beta}\exp\left(-\frac{1}{2}|x|^2\right).\] Due to this construction, $\varphi$ is a smooth function such that $\varphi(x) \geq \lambda_1(x)$ and $\varphi(x_0)=\lambda_1(x_0)$. Then, at the point $x_0$, we have \cite[Lemma 5]{BCD} that
\begin{align}
u_{kli}&=\varphi_i \delta_{kl}, \quad \text{for $1 \leq k,l \leq m$}, \label{1st order approximation} \\
\varphi_{ii}&\geq u_{11ii}+2\sum_{p>m} \frac{u_{1pi}^2}{\lambda_1-\lambda_p}. \label{2nd order approximation}
\end{align} 
In what follows, we will stop writing out the point $x_0$ of evaluation for convenience, but keep in mind that all the computations below are carried out at $x_0$.

Now, the function
\[\tilde{Q}(x):=(-u)^{\beta} \varphi(x) \exp \left(\frac{1}{2}|x|^2\right)=Q(x_0)\] has constant value at $x_0$. Taking \eqref{1st order approximation} and \eqref{2nd order approximation} into account, it follows that
\begin{align}
0&=\beta \frac{u_i}{u}+\frac{\varphi_i}{\varphi}+x_i \nonumber \\
&=\beta \frac{u_i}{u} + \frac{u_{11i}}{\lambda_1}+x_i \label{1st critical 1}
\end{align} and
\begin{align}
0&=\beta \frac{u_{ii}}{u}-\beta \frac{u_{i}^2}{u^2} + \frac{\varphi_{ii}}{\varphi}-\frac{\varphi_{i}^2}{\varphi^2}+1 \nonumber \\
&\geq \beta \frac{u_{ii}}{u}-\beta \frac{u_{i}^2}{u^2} + \frac{u_{11ii}}{\lambda_1}+2\sum_{p>m} \frac{u_{1pi}^2}{\lambda_1(\lambda_1-\lambda_p)}-\frac{u_{11i}^2}{\lambda_{1}^2}+1. \label{2nd critical 1}
\end{align}

\begin{remark}
When $m=1$, we could immediately obtain \eqref{1st critical 1} and \eqref{2nd critical 1} at $x_0$. The subsequent computations would work fine for both $m=1$ and $m \geq 2$.
\end{remark}

Contracting \eqref{2nd critical 1} with $F^{ii}=\sigma_{k}^{ii}$, we have
\begin{gather} \label{2nd critical 2}
\begin{split}
0 \geq &\ \frac{\beta k}{u} - \beta \sum_{i=1}^{n} \frac{F^{ii}u_{i}^2}{u^2}+\sum_{i=1}^{n} \frac{F^{ii}u_{11ii}}{\lambda_1}+2\sum_{i=1}^{n}\sum_{p>m} \frac{F^{ii}u_{1pi}^2}{\lambda_1(\lambda_1-\lambda_p)}\\
& -\sum_{i=1}^{n} \frac{F^{ii}u_{11i}^2}{\lambda_{1}^2}+\sum_{i=1}^{n} F^{ii},
\end{split}
\end{gather} where we have used Lemma \ref{sigma_k properties 1} to get
\[\sum_{i=1}^{n} F^{ii}u_{ii}=\sum_{i=1}^{n} \sigma_{k}^{ii}\lambda_i=k\sigma_k=k.\] Differentiating the equation $F(D^2u)=1$ twice, we obtain that
\begin{equation}
\sum_{i=1}^{n} F^{ii} u_{ii1}=0 \label{differentiate once}
\end{equation} and
\begin{equation}
\sum_{i,j,k,l} F^{ij,kl}u_{ij1}u_{kl1}+\sum_{i=1}^{n} F^{ii}u_{11ii}=0. \label{differentiate twice}
\end{equation} It follows from \eqref{differentiate twice} and \eqref{2nd derivatives for sigma k} that
\begin{align*}
\sum_{i=1}^{n} F^{ii}u_{11ii} &=-\sum_{i,j,k,l} F^{ij,kl}u_{ij1}u_{kl1} \\
&=-\sum_{p \neq q} F^{pp,qq} u_{pp1}u_{qq1} + \sum_{p \neq q} F^{pp,qq} u_{pq1}^2.
\end{align*} 
By \eqref{1st order approximation} and \eqref{2nd derivative for sigma k 2}, we may expand further that
\begin{gather} \label{expand 1}
\begin{split}
\sum_{p \neq q} \frac{F^{pp,qq}u_{pq1}^2}{\lambda_1} &\geq 2\sum_{i>m} \frac{F^{11,ii}u_{11i}^2}{\lambda_{1}}=2\sum_{i>m} \frac{F^{ii} - F^{11}}{\lambda_1(\lambda_1-\lambda_i)} u_{11i}^2.
\end{split}
\end{gather} Similarly, we have that
\begin{gather} \label{expand 2}
\begin{split}
2\sum_{i=1}^{n}\sum_{p>m} \frac{F^{ii}u_{1pi}^2}{\lambda_1(\lambda_1-\lambda_p)}& \geq 2\sum_{p>m} \frac{F^{pp}u_{1pp}^2}{\lambda_1(\lambda_1-\lambda_p)}+2\sum_{p>m}\frac{F^{11}u_{1p1}^2}{\lambda_1(\lambda_1-\lambda_p)}\\
&=2\sum_{i>m} \frac{F^{ii}u_{ii1}^2}{\lambda_1(\lambda_1-\lambda_i)} +2 \sum_{i>m} \frac{F^{11}u_{11i}^2}{\lambda_1(\lambda_1-\lambda_i)}.
\end{split}
\end{gather} Again, by \eqref{1st order approximation}, we have
\[u_{11i}=u_{1i1}=\delta_{1i} \cdot \varphi_{1}=0 \quad \text{for $1<i \leq m$}.\] Hence, we have
\begin{equation}
\sum_{i=1}^{n} \frac{F^{ii}u_{11i}^2}{\lambda_{1}^2}=\frac{F^{11}u_{111}^2}{\lambda_{1}^2}+\sum_{i>m} \frac{F^{ii}u_{11i}^2}{\lambda_{1}^2} \label{expand 3}
\end{equation} which is a trivial equality if $m=1$.

Substituting \eqref{expand 1}, \eqref{expand 2}, and \eqref{expand 3} back into \eqref{2nd critical 2}, we get
\begin{gather} \label{2nd critical 3}
\begin{split}
0\geq &\ -\sum_{p \neq q} \frac{F^{pp,qq} u_{pp1}u_{qq1}}{\lambda_1}-\frac{F^{11}u_{111}^2}{\lambda_{1}^2}+2\sum_{i>m} \frac{F^{ii}u_{ii1}^2}{\lambda_1(\lambda_1-\lambda_i)} \\
&\ +2\sum_{i>m} \frac{F^{ii} - F^{11}}{\lambda_1(\lambda_1-\lambda_i)} u_{11i}^2+2 \sum_{i>m} \frac{F^{11}u_{11i}^2}{\lambda_1(\lambda_1-\lambda_i)}-\sum_{i>m} \frac{F^{ii}u_{11i}^2}{\lambda_{1}^2}\\
&\ +\frac{\beta k}{u} - \beta \sum_{i=1}^{n} \frac{F^{ii}u_{i}^2}{u^2}+\sum_{i=1}^{n} F^{ii}.
\end{split}
\end{gather} Next, by the first order critical condition \eqref{1st critical 1}, we have
\begin{align*}
\beta \sum_{i=1}^{n} \frac{F^{ii}u_{i}^2}{u^2}&=\frac{1}{\beta}\sum_{i=1}^{n} F^{ii}\left(\frac{u_{11i}}{\lambda_1}+x_i\right)^2 \\
&\leq \frac{C}{\beta} \sum_{i=1}^{n} \frac{F^{ii}u_{11i}^2}{\lambda_{1}^2}+\frac{C}{\beta} \sum_{i=1}^{n} F^{ii}
\end{align*} for some $C>0$ depending on $\diam(\Omega)$. For notational convenience, we may set
\[\varepsilon:=\frac{C}{\beta}.\] By taking $\beta>0$ large, we will have $\varepsilon$ small. Thus, \eqref{2nd critical 3} becomes
\begin{gather} \label{2nd critical 4}
\begin{split}
0\geq &\ -\sum_{p \neq q} \frac{F^{pp,qq} u_{pp1}u_{qq1}}{\lambda_1}-(1+\varepsilon)\frac{F^{11}u_{111}^2}{\lambda_{1}^2}+2\sum_{i>m} \frac{F^{ii}u_{ii1}^2}{\lambda_1(\lambda_1-\lambda_i)} \\
&\ +2\sum_{i>m} \frac{F^{ii}}{\lambda_1(\lambda_1-\lambda_i)} u_{11i}^2-(1+\varepsilon)\sum_{i>m} \frac{F^{ii}u_{11i}^2}{\lambda_{1}^2}\\
&\ +\frac{\beta k}{u}+(1-\varepsilon)\sum_{i=1}^{n} F^{ii}.
\end{split}
\end{gather} We proceed to deal with \eqref{2nd critical 4} line by line.

For the first line in \eqref{2nd critical 4}, since 
\[u_{ii1}=u_{i1i}=\delta_{i1} \cdot \varphi_{i}=0, \quad \text{for $1<i \leq m$},  \quad \text{due to \eqref{1st order approximation}},\] we may apply Lemma \ref{the concavity inequality} to $\xi_i=u_{ii1}$ and obtain that
\begin{align*}
&\ -\sum_{p \neq q} \frac{F^{pp,qq} u_{pp1}u_{qq1}}{\lambda_1}-(1+\varepsilon)\frac{F^{11}u_{111}^2}{\lambda_{1}^2}+2\sum_{i>m} \frac{F^{ii}u_{ii1}^2}{\lambda_1(\lambda_1-\lambda_i)}\\
\geq &\ \left[\gamma-(1+\varepsilon)\right]\frac{F^{11}u_{111}^2}{\lambda_{1}^2} - \frac{2}{F\lambda_1}\left(\sum_{i=1}^{n} F^{ii}u_{ii1}\right)^2.
\end{align*}
Assume $2k>n$. Then according to \eqref{gamma}, the number $\gamma$ can be taken as 
\[\gamma=1+\delta\] for some $\delta=\delta(n)>0$ depending only on $n$. Therefore, by taking a small enough $\varepsilon$ and invoking \eqref{differentiate once}, we can make
\begin{align*}
-\sum_{p \neq q} \frac{F^{pp,qq} u_{pp1}u_{qq1}}{\lambda_1}-(1+\varepsilon)\frac{F^{11}u_{111}^2}{\lambda_{1}^2}+2\sum_{i>m} \frac{F^{ii}u_{ii1}^2}{\lambda_1(\lambda_1-\lambda_i)} \geq 0
\end{align*}
\begin{remark}
When $m=1$, this non-negativity holds more directly by applying Lemma \ref{the concavity inequality}.
\end{remark}
For the second line in \eqref{2nd critical 4}, we compute directly,
\begin{align*}
&\ 2\sum_{i>m} \frac{F^{ii}}{\lambda_1(\lambda_1-\lambda_i)} u_{11i}^2-(1+\varepsilon)\sum_{i>m} \frac{F^{ii}u_{11i}^2}{\lambda_{1}^2} \\
=&\ \sum_{i>m} \frac{2\lambda_1-(1+\varepsilon)(\lambda_1-\lambda_i)}{\lambda_1-\lambda_i}\frac{F^{ii}u_{11i}^2}{\lambda_{1}^2}\\
=&\ \sum_{i>m} \frac{(1-\varepsilon)\lambda_1+(1+\varepsilon)\lambda_i}{\lambda_1-\lambda_i} \frac{F^{ii}u_{11i}^2}{\lambda_{1}^2}.\\
\end{align*} Assume $2k>n$. By Lemma \ref{sigma_k properties 2}, we can make the coefficient
\[\frac{(1-\varepsilon)\lambda_1+(1+\varepsilon)\lambda_i}{\lambda_1-\lambda_i} \geq \frac{(2k-n)-n\varepsilon}{n}\] non-negative by taking $\varepsilon=\varepsilon(n,k)>0$ small enough. Thus, the second line has been eliminated from \eqref{2nd critical 4}.

Now, only the third line remains in \eqref{2nd critical 4}. That is,
\begin{equation}
0 \geq \frac{\beta k}{u} + (1-\varepsilon)\sum_{i=1}^{n} F^{ii}. \label{2nd critical 5}
\end{equation}

Finally, by taking $\varepsilon$ small, we may have a term of 
\[(1-\varepsilon)\sum_{i=1}^{n} F^{ii} \geq \frac{1}{2}\sum_{i=1}^{n} F^{ii}.\] By Lemma \ref{sigma_k properties 1} and Lemma \ref{NM}, we have
\[\sum_{i=1}^{n} F^{ii}=(n-k+1)\sigma_{k-1} \geq C(n,k) \sigma_{1}^{\frac{1}{k-1}} \sigma_{k}^{\frac{k-2}{k-1}} \geq C(n,k) \lambda_{1}^{\frac{1}{k-1}}.\] Therefore, \eqref{2nd critical 5} implies that
\[0 \geq -\frac{\beta k}{(-u)}+C(n,k)\lambda_{1}^{\frac{1}{k-1}}\] and the desired estimate readily follows.

The proof is now complete.
\end{proof}
\bibliography{refs}
\end{document}